\documentclass[12pt]{article}
\usepackage{amsfonts,amssymb,amsmath,amsthm}
\usepackage{graphicx}
\usepackage{epstopdf}
\ifpdf
  \DeclareGraphicsExtensions{.eps,.pdf,.png,.jpg}
\else
  \DeclareGraphicsExtensions{.eps}
\fi
\usepackage{pbox}
\usepackage{color,url}
\usepackage{dsfont}
\usepackage{setspace}
\usepackage{graphicx}
\usepackage{epsfig}
\usepackage{float}
\usepackage{stmaryrd}
\usepackage{pifont}
\usepackage{comment}
\usepackage{enumerate}
\usepackage[section]{placeins}

\newtheorem{thm}{\bf{Theorem}}[section]

\newtheorem{note}[thm]{\bf{Note}}
\newtheorem{df}[thm]{\bf{Definition}}
\newtheorem{cor}[thm]{\bf{Corollary}}

\newtheorem{prop}[thm]{\bf{Proposition}}

\newtheorem{ex}[thm]{\bf{Example}}

\newcommand{\dom}{\operatorname{dom}}

\newcommand{\Id}{\operatorname{Id}}

\newcommand{\rank}{\operatorname{rank}}

\newcommand{\R}{\operatorname{\mathbb{R}}}
\newcommand{\N}{\operatorname{\mathbb{N}}}
\newcommand{\X}{\operatorname{\mathcal{X}}}

\newcommand{\y}{\operatorname{\mathcal{Y}}}
\newcommand{\m}{\operatorname{\mathcal{M}}}
\newcommand{\s}{\operatorname{\mathcal{S}}}

\newcommand{\A}{\operatorname{\mathcal{A}}}

\newcommand{\bbm}{\begin{bmatrix}}
\newcommand{\ebm}{\end{bmatrix}}

\newcommand{\ds}{\delta_f}

\newcommand{\dsf}{\delta_f}

\newcommand{\dnsf}{\delta_{\ns f}}

\newcommand{\ns}{\nabla_s}
\newcommand{\sh}{\nabla^2_s}

\newcommand{\h}{\mathbf{H}}
\newcommand{\hba}{\overline{\mathbf{H}}}
\newcommand{\hhat}{\widehat{\mathbf{H}}}
\newcommand{\zero}{\mathbf{0}}
\newcommand{\az}{\alpha_0}
\newcommand{\al}{\alpha}
\newcommand{\abar}{\overline{\alpha}}
\newcommand{\oneh}{\frac{1}{2}}
\newcommand{\Y}{\mathcal{Y}}
\newcommand{\one}{\mathbf{1}}

\newcommand\scalemath[2]{\scalebox{#1}{\mbox{\ensuremath{\displaystyle #2}}}}

\begin{document}
\title{Hessian approximations}
\author{Warren Hare\thanks{Department of Mathematics, University of British Columbia, Okanagan Campus, Kelowna, B.C. V1V 1V7, Canada. Research partially supported by NSERC of Canada Discovery Grant 2018-03865. warren.hare@ubc.ca, ORCID 0000-0002-4240-3903}\and Gabriel Jarry--Bolduc\thanks{Department of Mathematics, University of British Columbia, Okanagan Campus, Kelowna, B.C. V1V 1V7, Canada. Research partially supported by Natural Sciences and Engineering Research Council (NSERC) of Canada Discovery Grant 2018-03865. gabjarry@alumni.ubc.ca }\and Chayne Planiden\thanks{School of Mathematics and Applied Statistics, University of Wollongong, Wollongong, NSW, 2500, Australia. Research supported by University of Wollongong. chayne@uow.edu.au, ORCID 0000-0002-0412-8445}}
\maketitle\author

\begin{abstract}
This work introduces the {\em nested-set Hessian} approximation, a second-order approximation method that can be used in any derivative-free optimization routine that requires such information. It is built on the foundation of the generalized simplex gradient and proved to have an error bound that is on the order of the  maximal radius of the two sets used in its construction. We show that when the points used in the computation of the nested-set Hessian have a favourable structure, $(n+1)(n+2)/2$ function evaluations are sufficient to approximate the Hessian.  However, the nested-set Hessian also allows for evaluation sets with more points without negating the error analysis.  Two calculus-based approximation techniques of the Hessian are developed and some advantages of the same are demonstrated.
\end{abstract}



\section{Introduction}

Derivative-free optimization (DFO) is the study of finding the minimum value and minimizers of a function without using gradients or higher-order derivative information \cite{audet2017derivative}. DFO is gaining in popularity in recent years and is useful in cases where gradients are unavailable, computationally expensive, or simply difficult to obtain \cite{audet2017derivative,conn2009introduction}.  DFO methods have been used in both smooth \cite{berghen2005condor,cocchi2018implicit,Gratton2017b,liuzzi2019trust,powell2003trust,powell2009bobyqa,shashaani2018astro,wild2013global} and nonsmooth \cite{audet2018algorithmic,bagirov2008discrete,hare2013derivative,hare2016proximal,larson2016manifold,MMSMW2017} optimization, and most commonly takes the form of either direct-search \cite{amaioua2018efficient, audet2008nonsmooth,audet2018mesh,BBN2018,Gratton2017direct} or  model-based \cite{hare2019derivative,liuzzi2019trust,Maggiar2018,verderio2017construction,wild2013global} methods, while some use a blend of both \cite{AudetIanniLeDigTribes2014,MR2457346}.

Model-based DFO methods use numerical analysis techniques to  approximate gradients and Hessians in a manner that has controllable error bounds. One example of a technique to approximate gradients is the  simplex gradient. This method consists of taking a set of $n+1$ sample points in $\R^n$ that are appropriately spaced, called a simplex, and using them to build a linear interpolation function that approximates the objective function locally, then calculating the gradient of the affine function. The simplex gradient has been shown to have a nicely bounded error, for functions whose gradients are locally Lipschitz continuous \cite{bortz1998simplex}. 

One recent line of research explores methods of improving or generalizing the simplex gradient for its use in DFO \cite{MR3935094,MR4074016,hare2020error,MR3348587}.  In \cite{MR4074016}, the simplex gradient was generalized so as not to require exactly $n+1$ points in $n$-dimensional space; an error-controlled approximation can now be found using any finite number of properly-spaced points. In \cite{hare2020error}, a similar approach was used to examine the generalized centred simplex gradient, an approximation that uses twice as many points, but results in an improvement on the error control from order-$1$ to order-$2$  \cite{hare2020discussion}. All three papers \cite{MR4074016,hare2020error,MR3348587} include a study of calculus rules as they can be applied to the numerical approximations that they examine.  In this work, we continue the development of these tools by introducing the \emph{nested-set Hessian}, thereby presenting the first such results for Hessian approximations.

The nested-set Hessian provides an accurate approximation of the Hessian of objective function $f:\R^n\to\R$ under reasonable assumptions. The technique uses a point of interest $x^0$ and two sets of directions $S$ and $T$ that total $n(n+1)$ points in the domain of $f$.  However, these points need not be distinct, so if we make careful choices, then $\frac{1}{2}(n+1)(n+2)$ function evaluations is sufficient to compute the nested-set Hessian. Defining $\Delta_S,\Delta_T$ as the radii of $S$ and $T$, respectively, we prove that if $S$ and $T$ are full row rank and the true Hessian $\nabla^2f$ exists and is Lipschitz continuous, then the nested-set Hessian is an accurate estimate of $\nabla^2f$ to within a multiple of $\Delta_u:=\max\{\Delta_S,\Delta_T\}$.  In the language of order-$N$ accuracy \cite{hare2020discussion}, the nested-set Hessian provides order-$1$ Hessian approximations (see Definition \ref{df:orderNaccuracy} herein). 

After developing the error analysis, we explore calculus rules for the nested-set Hessian, forming a product rule, quotient rule and power rule with error bounds. We provide two methods of obtaining these bounds, which we call the simplex calculus Hessian and the quadratic calculus Hessian. 

The remainder of this paper is organized as follows. Section \ref{sec:prelim} contains a description of notation and some needed definitions, including those of the generalized simplex gradient and the nested-set Hessian. In Section \ref{sec:qishc}, we establish the minimal poised set for nested-set Hessian computation and show that it is well-poised for quadratic interpolation. We construct the quadratic interpolation function in Section \ref{sec:qishc}. Section \ref{sec:error} develops the nested-set Hessian error bound, which we use in Section \ref{sec:calculus} to establish the product rule, quotient rule and power rule error bounds. Section \ref{sec:conc} contains concluding remarks and recommends areas of future research in this vein.

\section{Preliminaries}\label{sec:prelim}

Throughout this work, we use the standard notation found in \cite{rockwets}. The domain of a function $f$ is denoted by $\dom f$. The transpose of a matrix $A$ is denoted by $A^\top$. We work in finite-dimensional space $\R^n$ with inner product $x^\top y=\sum_{i=1}^nx_iy_i$ and induced norm $\|x\|=\sqrt{x^\top x}$. The identity matrix in $\R^{n \times n}$ is denoted by $\Id_n$. We use $e^i \in \R^n$ for $i \in \{1, 2, \dots, n\}$ to denote the standard unit basis vectors in $\R^n$, i.e.\ the $i$\textsuperscript{th} column of $\Id_n.$ The zero vector in $\R^n$ is denoted $\zero$ and the zero matrix in $\R^{n \times n}$ is denoted $\zero_{n \times n}$. The entry in the $i$\textsuperscript{th} row and $j$\textsuperscript{th} column of a matrix $A$ is denoted $A_{i,j}.$ Given a matrix $A \in \R^{n \times m},$ we use the induced matrix norm 
\begin{align*}
    \Vert A \Vert=\Vert A \Vert_2=\max \{ \Vert Ax\Vert_2 \, : \, \Vert x \Vert_2=1 \}
\end{align*}
and the Frobenius norm 
\begin{align*}
    \Vert A \Vert_F=\left (\sum_{i=1}^n \sum_{j=1}^m A_{i,j}^2 \right )^{\frac{1}{2}}.
\end{align*}
We denote by $B(x^0,\overline{\Delta})$ the open ball centred about $x^0$ with radius $\overline{\Delta}$. We define a quadratic function $Q: \R^n \to \R$ to be a function of the form $Q(x)=\alpha_0+\alpha^\top x+\frac{1}{2}x^\top \h x$ where $\alpha_0 \in \R, \alpha \in \R^n$ and $\h=\h^\top \in \R^{n \times n}.$ An affine function $\mathcal{L}:\R^n \to \R$ is defined as any function that can be written in the form $\mathcal{L}(x)=\alpha^\top x+\alpha_0$ where $\alpha \in \R^n$ and $\alpha_0 \in \R.$ Therefore, affine functions and constant functions $C(x)=\alpha_0$ are also considered quadratic functions. Next, we introduce fundamental definitions and notation that will be used throughout this paper.
\begin{df}[Poised for quadratic interpolation]\emph{\cite[Definition 9.8]{audet2017derivative}} \label{def:poisedqi}
The set of distinct points $\Y=\{ y^0, y^1, \dots, y^m \} \subset \R^n$ with $m=\frac{1}{2}(n+1)(n+2)-1$ is poised for quadratic interpolation if the system
\begin{align} 
    \alpha_0+\alpha^\top y^i +\frac{1}{2} (y^i)^\top \h y^i=\zero, \quad i \in \{0, 1, \dots, m\},  \label{eq:poisedqisys}
\end{align}
has a unique solution for $\alpha_0 \in \R, \alpha \in \R^n,$ and $\h=\h^\top \in \R^{n \times n}.$
\end{df}
\begin{df}[Quadratic interpolation function]\emph{\cite[Definition 9.9]{audet2017derivative}} \label{def:quadinterpolationfunc}
Let $f:\dom f\subseteq \R^n \to \R$ and let $\Y=\{y^0, y^1, \dots, y^m\} \subset \dom f$ with $m=\frac{1}{2}(n+1)(n+2)-1$ be poised for quadratic interpolation. Then the quadratic interpolation function of $f$ over $\Y$ is
\begin{align*}
    Q_f(\y)(x)&=\alpha_0+\alpha^\top x+\frac{1}{2}x^\top \h x,
\end{align*}
where $(\alpha_0, \alpha, \h=\h^\top)$ is the unique solution to
\begin{align*}
    \alpha_0+\alpha^\top y^i+\frac{1}{2}(y^i)^\top \h y^i&= f(y^i), \quad i \in \{0, 1, \dots, m\}.
\end{align*}
\end{df}

In the next definition, we introduce key notation used in the construction of the nested-set Hessian.  Within, we write a set of vectors in matrix form, by which we mean that each column of the matrix is a vector in the set.

\begin{df}[Hessian notation]
Let $f: \dom f \subseteq \R^n \to \R$ and let $x^0 \in \dom f$ be the \emph{point of interest}.  Let  
\begin{align*}
    S&=\bbm s^1&s^2&\cdots&s^m\ebm \in \R^{n \times m} ~~\mbox{and}\\
    T&=\bbm t^1&t^2&\cdots&t^k \ebm \in \R^{n \times k}
\end{align*}
be two sets of directions  contained  in $\R^n$, written in matrix form  such that $x^0+s^i,$ $x^0+t^j, x^0+s^i+t^j \in \dom f$ for all $i \in \{1, 2, \dots, m\}$ and  for all $j \in \{1, 2, \dots, k\}$. Define 
    $$\Delta_S=\max\limits_{i \,\in\{1,\ldots,m\}}\|s^i \|, \quad \Delta_T=\max\limits_{j \,\in\{1,\ldots,k\}}\|t^j\|,$$
and
\begin{align*}
\dsf\left (x^0;T \right )&=\left[\begin{array}{c}f(x^0+t^1)-f(x^0)\\f(x^0+t^2)-f(x^0) \\\vdots\\f(x^0+t^k)-f(x^0)\end{array}\right]\in\R^k.
\end{align*}
\end{df}

Recall that for nonsquare matrices, a generalization of the matrix inverse is the pseudoinverse. The most well-known type of matrix pseudoinverse is the Moore--Penrose pseudoinverse.

\begin{df}[Moore--Penrose pseudoinverse] \label{def:mpinverse}
Let $A\in\R^{n\times m}$. The \emph{Moore--Penrose pseudoinverse} of $A$, denoted by $A^\dagger$, is the unique matrix in $\R^{m \times n}$ that satisfies the following four equations:
\begin{align*}
AA^\dagger A&=A,\\A^\dagger AA^\dagger&=A^\dagger,\\(AA^\dagger)^\top&=AA^\dagger,\\(A^\dagger A)^\top&=A^\dagger A.
\end{align*}
\end{df}

\noindent Note that given $A \in \R^{n \times m},$ there exists a unique Moore--Penrose pseudoinverse $A^\dagger\in\R^{m\times n}.$ The following two properties hold.
\begin{enumerate}[(i)]
\item If $A$ has full column rank $m$, then $A^\dagger$ is a left-inverse of $A$, that is $A^\dagger A=\Id_m.$ In this case, $A^\dagger=(A^\top A)^{-1} A^\top.$ 
\item If $A$ has full row rank $n$, then $A^\dagger$ is a right-inverse of $A$, that is $AA^\dagger=\Id_n.$ In this case, $A^\dagger=A^\top (A A^\top)^{-1}.$
\end{enumerate}

Before introducing the nested-set Hessian, it is valuable to recall the definition of the generalized simplex gradient and its corresponding error bound.

\begin{df}[Generalized simplex gradient] \emph{\cite[Definition 2]{hare2020error}} Let $f:\dom f \subseteq \R^n\to\R$.  Let $x^0 \in \dom f$ be the point of interest.  Let $T=\bbm t^1&t^2&\cdots&t^k\ebm \in \R^{n \times k}$ with $x^0+t^j \in \dom f$ for all $j\in\{1,2,\ldots,k\}$. The \emph{generalized simplex gradient} of $f$ at $x^0$ over  $T$ is denoted by $\ns f(x^0;T)$ and defined by
\begin{equation*}\label{eq:simplex}\ns f(x^0;T)=(T^\top)^\dagger\ds (x^0;T).\end{equation*}
\end{df}

The Moore--Penrose pseudoinverse allows for $T$ to contain any number of columns, instead of being restricted to $k=n$. When $T$ is in $\R^{n \times n}$ and $\rank T=n$, then the Moore--Penrose pseudoinverse of $S^\top$ is just the inverse of $S^\top$ and we recover the definition of the simplex gradient \cite[Definition 9.5]{audet2017derivative}.

\begin{prop}[Error bound for generalized simplex gradient]\emph{\cite[Proposition 3]{hare2020error}}\label{prop:GSGerror}
 Let $T=\bbm t^1&t^2&\cdots&t^k\ebm \in \R^{n \times k}$ have full row rank. Let $f:\dom f \subseteq \R^n\to\R$ be $\mathcal{C}^{2}$ on $B(x^0;\overline{\Delta})$ where $x^0 \in \dom f$ is the point of interest and $\overline{\Delta}>\Delta_T$. Assume $x^0+t^j \in \dom f$ for all $j\in\{1,2,\ldots,k\}$. Denote by $L_{\nabla f}$ the Lipschitz constant of $\nabla f$ on $B(x^0,\overline{\Delta}).$ 
 Then
\begin{align}\label{eq:ebgsg}
\|\nabla_s  f(x^0;T)-\nabla f(x^0)\| &\leq \frac{\sqrt{k}}{2}L_{\nabla f}  \left \Vert (\widehat{T}^\top)^\dagger \right \Vert \Delta_T,
\end{align}
where $\widehat{T}=T/\Delta_T.$
\end{prop}

 The error bound presented in Proposition \ref{prop:GSGerror} requires $T$ to have full row rank, and therefore covers the {\em determined} ($k=n$) and {\em overdetermined} ($k>n$) cases.  The underdetermined case ($k<n$ and $\rank T=k$) has also been studied \cite {MR3348587}, but is not used in this paper.

We now introduce the key definition for this paper, the nested-set Hessian.  It requires two finite sets of directions $S$ and $T$, and a point of interest, i.e., the point where the Hessian is approximated.  Similar to the generalized simplex gradient, it involves the Moore--Penrose pseudoinverse of $S$  and a difference matrix.  In this case, the difference matrix consists of the difference between generalized simplex gradients.

\begin{df}[Nested-set Hessian]\label{def:gsh}
Let $f:\dom f \subseteq \R^n \to \R$ and let $x^0 \in \dom f$ be the point of interest.  Let $S=\bbm s^1&s^2&\cdots &s^m \ebm  \in \R^{n \times m}$ and $ T=\bbm t^1&t^2&\cdots&t^k \ebm \in \R^{n \times k}$ with $x^0+s^i, x^0+t^j, x^0+s^i+t^j \in \dom f$ for all $i \in \{1, 2, \dots, m\}$ and for all $j \in \{1, 2, \dots, k\}$. The \emph{ nested-set Hessian}  of $f$ at $x^0$ over $S$ and $T$ is denoted by $\sh f(x^0;S,T)$ and defined by
\begin{equation*}\label{eq:simplexhess}\sh f(x^0;S,T)=(S^\top)^\dagger\dnsf (x^0;S,T),\end{equation*}
where
\begin{equation*}
    \dnsf(x^0;S,T)=\left[\begin{array}{c}(\ns f(x^0+s^1;T)-\ns f(x^0;T))^\top\\(\ns f(x^0+s^2;T)-\ns f(x^0;T))^\top\\\vdots\\(\ns f(x^0+s^m;T)-\ns f(x^0;T))^\top\end{array}\right]\in \R^{m \times n}.
\end{equation*}
\end{df}
Next we recall the definition of \emph{order-N Hessian accuracy}, as it is used in this article several times to describe the quality of the Hessian  approximation techniques developed.
\begin{df}[Order-N  Hessian accuracy] \label{df:orderNaccuracy}
Given $f \in \mathcal{C}^2$, $x^0 \in \dom f$ and $\overline{\Delta}>0,$ we say that $\{\tilde f_\Delta\}_{\Delta \in (0,\overline{\Delta}]}$ is a class of models of $f$ at $x^0$  parameterized by $\Delta$ that provides order-N Hessian accuracy at $x^0$ if there exsits a scalar $\kappa(x^0)$ such that, given any $\Delta \in (0, \overline{\Delta}],$ the model $\tilde{f}_\Delta$ satisfies
$$\Vert \nabla^2 f(x^0)-\nabla^2 \tilde{f}_\Delta (x^0)\Vert\leq \kappa(x^0)\Delta^N.$$
\end{df}
\begin{df}
The set of all distinct points utilized in the computation of $\nabla_s^2 f(x^0;S,T)$ is said to be the set for {\em nested-set Hessian computation} (NSHC) and is denoted $\s(x^0;S,T).$ 
\end{df} 

Note that $\s(x^0;S,T)$ contains at most $(m+1)k$ distinct points, but can contain fewer points if some of them overlap. 

From Definition \ref{def:gsh}, we see that the only condition for $\nabla_s^2 f(x^0;S,T)$  to be well-defined is that all the points used in the computation of the matrix $\delta_{\nabla_s f}(x^0;S,T)$ are in $\dom f.$  In the next section, we investigate the possibility of choosing $S$ and $T$ so that $\s(x^0;S,T)$ contains a minimal number of points, while providing a good approximation of the Hessian. Reducing the number of distinct points in  $\s(x^0;S,T)$ is extremely valuable, as it will decrease the number of distinct function evaluations necessary to compute $\nabla^2_s f(x^0;S,T).$

\section{Minimal poised set for NSHC and quadratic interpolation}\label{sec:qishc}

In this section, we show that if  the sets $S$ and $T$ have a specific structure, then the number of distinct function evaluations necessary to compute the nested-set Hessian is $(n+1)(n+2)/2,$ i.e., $\s(x^0;S,T)$ contains exactly $(n+1)(n+2)/2$ distinct points.  We then explore some results that occur when such a structure is used.
We begin with the definition of \emph{minimal poised set} for NSHC. 

\begin{df}[Minimal poised set for NSHC]\label{df:minimalNSHC}
Let  $x^0 \in \R^n$ be the point of interest.  Let $S=\bbm s^1&s^2&\cdots&s^n\ebm \in \R^{n \times n}$ and $T = \bbm t^1 & t^2 & \cdots & t^n \ebm \in \R^{n \times n}$.  We say that $\s(x^0;S,T)$ is a \emph{minimal poised set for NSHC} at $x^0$ if and only if $S$ and $T$ are full rank and $\s(x^0;S,T)$ contains exactly $(n+1)(n+2)/2$ distinct points.
\end{df}

This definition requires $S$ and $T$ to have exactly $n$ columns and be full rank.  This implies that $S$ and $T$ are the minimal size to ensure that Proposition \ref{prop:GSGerror} applies to the simplex gradients constructed in NSHC.  In Section \ref{sec:error}, we shall see that the nested-set Hessian computed over a minimal poised set also satisfies the assumptions of the error bound in Theorem \ref{prop:ebsimplexhessian}
We next show that it is possible to create a minimal poised set for NSHC.

\begin{prop} \label{prop:reducingfe}
Let $x^0$ be the point of interest.  Let $S=\bbm s^1&s^2&\cdots&s^n\ebm \in \R^{n \times n}$.  Define the set $U_k$ for each index $k\in\{0,1,\ldots,n\}$ as    $$U_0=S$$
and 
    $$U_k=\bbm s^1-s^k&s^2-s^k&\cdots&s^{k-1}-s^k&-s^k&s^{k+1}-s^k&\cdots&s^n-s^k\ebm \in \R^{n \times n},k\neq0.$$  Then for each $k$, $|\s(x^0;S,U_k)| \leq (n+1)(n+2)/2$.  Moreover, if $S$ has full rank, then $|\s(x^0;S,U_k)| = (n+1)(n+2)/2$.
\end{prop}

\begin{proof} Without loss of generality, let $x^0=\zero$. First, suppose $k\in\{1, 2, \dots, n\}.$  For arbitrary function $f$, consider the matrix $\delta_{\nabla_s f}(x^0;S,U_k)$. The computation of $\nabla_s f(x^0;U_k)$ evaluates $f$ at the points  
    $$\begin{array}{l}
    ~~~\{\zero, (s^1-s^k), \dots,  (s^{k-1}-s^k), -s^k,  (s^{k+1}-s^k),\dots, (s^n-s^k)\}\\
    =\{\zero\}\cup\{-s^k\}\cup\left\{s^{i}-s^k \right\}_{i \neq k}.
    \end{array}$$
For $i \neq k$, the computation of $\nabla_s f(x^0+s^i;U_k)$ evaluates $f$ at the points 
    $$\begin{array}{l}
    ~~~\{s^i, s^i+(s^1-s^k), \dots, s^i+(s^{k-1}-s^k), s^i-s^k,  s^i+(s^{k+1}-s^k),\dots, s^i+(s^n-s^k)\}\\
    =\{s^i\}\cup\{s^i-s^k\}\cup\{s^i + s^j - s^k\}_{j\neq k}.
    \end{array}$$
The computation of $\nabla_s f(x^0+s^k;U_k)$ evaluates $f$ at the points 
   $$\begin{array}{l}
    ~~~\{s^k, s^k+(s^1-s^k), \dots, s^k+(s^{k-1}-s^k), s^k-s^k,  s^k+(s^{k+1}-s^k),\dots, s^k+(s^n-s^k)\}\\
    =\{s^k\}\cup\{\zero\}\cup\{s^i\}_{i\neq k} =\{\zero\}\cup S.
    \end{array}$$
Thus, $f$ is evaluated at the points  
    \begin{equation}\label{eq:listofpoints}\begin{array}{l} ~~\bigg(\{\zero\}\cup\{-s^k\}\cup\left\{s^{i}-s^k \right\}_{i \neq k}\bigg) 
    ~ \cup ~
    \bigg(\{s^i\}\cup\{s^i-s^k\}\cup\{s^i + s^j - s^k\}_{j\neq k}\bigg)
    ~ \cup ~
    \bigg(  \{\zero\}\cup S \bigg)\\
    = \{\zero\} \cup S \cup\{-s^k\}\cup\left\{s^{i}-s^k \right\}_{i \neq k} \cup  (\{s^i + s^j - s^k\}_{i=1}^n)_{j>i}^n.
    \end{array}\end{equation}
This is at most $1+n+1+(n-1)+(n)(n-1)/2 = (n+1)(n+2)/2$ points.

Now, suppose $k=0.$ Using a similar process to the above, we find that $f$ is evaluated at the points   
\begin{equation}\label{eq:listofpointsu0}\begin{array}{l} \bigg(\{\zero\}\cup\left\{s^{i} \right \}_{i=1}^n \bigg) 
    ~ \cup ~
    \bigg(\{2s^i\}_{i=1}^n\cup\{s^i+s^j\}_{i \neq j}\bigg).
    \end{array}\end{equation}
This is at most $(n+1)(n+2)/2$ points.

Finally, if $S$ is full rank, then the four sets in  \eqref{eq:listofpoints} and the four sets in \eqref{eq:listofpointsu0} are disjoint, so we have exactly $ (n+1)(n+2)/2$ function evaluations.\qedhere
\end{proof}

Using $S=\Id_n$ in Proposition \ref{prop:reducingfe}, we can create $n+1$ canonical minimal poised sets for NSHC.
\begin{df}[$k^{\mbox{th}}$-canonical minimal poised set for NSHC] \label{def:canminset}
Let $x^0 \in \R^n$ be the point of interest.  Let $S=\Id_n$. Fix $k \in \{0, 1, \dots, n\}$. Let  
    $$E_0= \Id_n$$
and
    $$E_k=\bbm e^1-e^k&e^2-e^k&\cdots&e^{k-1}-e^k&-e^k&e^{k+1}-e^k&\cdots&e^n-e^k\ebm, k\neq0.$$ 
Then $\s(x^0;\Id_n,E_k)$ is  called the {\em $k^{\mbox{th}}$-canonical minimal poised set for NSHC} at $x^0$. 
\end{df}

From Proposition \ref{prop:reducingfe} and the fact that any matrix $E_k$ is full rank in Definition \ref{def:canminset}, the $k^{\mbox{th}}$-canonical minimal poised set for NSHC is indeed a minimal poised set for NSHC. Henceforth, we use the notation $\m(x^0;S,U_k)$ to denote a minimal poised set for NSHC at $x^0$ that takes the form constructed in Proposition \ref{prop:reducingfe}.

Note that the order of the directions in $S$ and $T$ are arbitrary.  Thus, it is immediately clear that if $\s(x^0;S,T)$ is a minimal poised set for NSHC and $P_1, P_2 \in \R^{n\times n}$ are permutation matrices, then  $\s(x^0;SP_1,TP_2)$ is also a minimal poised set for NSHC at $x^0$.  The next proposition expands this idea and demonstrates how to construct minimal poised sets for NSHC. 

\begin{prop}\label{prop:invmatrix}
Let $x^0 \in \R^n$ be the point of interest. Let $S,T\in\R^{n \times n}$.  Let $N \in \R^{n \times n}$ be an invertible matrix and $P_1, P_2\in \R^{n \times n}$ be permutation matrices.  Define $\overline{S}=NSP_1$ and $\overline{T}=N T P_2.$ Then  $\s(x^0;S,T)$ is a minimal poised set for NSHC at $x^0$ if and only if $\s(x^0;\overline{S},\overline{T})$ is a minimal poised set for NSHC at $x^0$.
\end{prop}

\begin{proof} The proof follows trivially from properties of matrices.\end{proof}

It follows that  $\s(x^0;S,T)$ is a minimal poised set for NSHC at $x^0$ if and only if the set $\s(x^0;\beta SP_1, \beta T P_2)$ is a minimal poised set for NSHC at $x^0$, where $\beta$ is a nonzero scalar and $P_1, P_2$ are permutaion matrices.
\begin{ex}
Let $x^0=(0,0)$. The $2^{\mbox{nd}}$-canonical minimal poised set for NSHC in $\R^2$ contains the points $(0,-1),(0,0),(0,1),(1,-1),(1,0)$ and $(2,-1)$.  In this case, $S = \{e^1, e^2\}$ and $T=\{(e^1-e^2), -e^2\}$.  Figure \ref{fig:canset} illustrates this set.  

\begin{figure}[ht]
\caption{The $2^{\mbox{nd}}$-canonical minimal poised set for NSHC at $x^0$ in $\R^2$.} \label{fig:canset}
\includegraphics[scale=0.5]{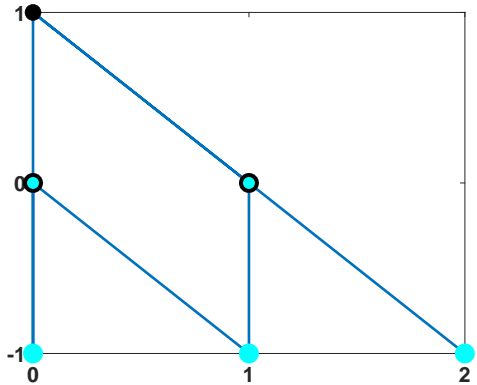}
\centering
\end{figure}

The points in $\{(0,0)\} \cup \{(0,0) + S\} = \{0, e^1, e^2\}$ are represented with solid black borders.  These are the base points where simplex gradients will be computed.  The lines represent the vectors corresponding to $T$ emanate from $\{(0,0)\} \cup \{(0,0) + S\}$.  The points in $\left(\{(0,0)\} \cup \{(0,0) + S\}\right) + T$ are represented with cyan cores.  These are the points used to construct the simplex gradients.   Notice the points $(0,0)$ and $(1,0)$ have both black borders and cyan cores.  These are the common points that allow the number of function evaluations to be reduced to $(n+1)(n+2)/2=6$.
\end{ex}

We next demonstrate that every minimal poised set for NSHC of the form $\m(x^0;S,U_k)$ is poised for quadratic interpolation.  We then show that the converse is not true; it is possible to construct a set that is poised for quadratic interpolation, but does not take the form of $\m(x^0;S,U_k)$.  

\begin{prop} \label{prop:minimalsetispoisedforqi} Let $S \in \R^{n \times n}$.  Select $k \in \{0, 1, \ldots, n\}$ and define $U_k$ as in Proposition \ref{prop:reducingfe}.  Then $\m(x^0;S,U_k)$ is poised for quadratic interpolation.
\end{prop}
\begin{proof}  Noting that $\m(x^0;S,U_k)$ is poised for quadratic interpolation if and only if the set  $\m(x^0;S,U_k) -x^0$ is poised for quadratic interpolation, we assume without loss of generality that $x^0=\zero$.

{\bf Case I: $k\in \{1, 2, \ldots, n\}$.}  Without loss of generality, by Proposition \ref{prop:invmatrix}, assume that $k=n$. 
The points contained in $\m(x^0;S,U_n)$ are
    $$\{\zero\} \cup S \cup\{-s^k\}\cup\left\{s^{i}-s^k \right\}_{i \neq k} \cup  (\{s^i + s^j - s^n\}_{i=1}^n)_{j>i}^n.
   $$
We show that using this set, the only solution  to \eqref{eq:poisedqisys} is the trivial solution.  
Considering the point $\zero$, we obtain
\begin{align}
    \az&=0. \label{eq:x0}
\end{align}
Considering the points $s^i$ for $i \in \{1, 2, \dots, n\}$ and noting that for all $i \in \{1, 2, \dots, n\},$ $s^i=Se^i$, we obtain
\begin{align}
    \abar^\top e^i+ \oneh (e^i)^\top \hhat e^i &=0, \label{eq:x0h}
\end{align}
where $\abar^\top=\al^\top S$ and $\hhat=S^\top \h S.$ Note that $\hhat$ is symmetric.  
Considering the points $s^i-s^n$ for $i \in \{ 1, 2, \dots, n-1\},$ we obtain
\begin{equation*}
     \abar^\top e^i -\abar^\top e^n -(e^i)^\top \hhat e^n + \oneh (e^i)^\top \hhat e^i + \oneh (e^n)^\top \hhat e^n=0.
\end{equation*}
Using \eqref{eq:x0h}, this simplifies to 
\begin{align}
  -\abar^\top e^n -(e^i)^\top \hhat e^n + \oneh (e^n)^\top \hhat e^n=0. \label{eq:x0eien}
\end{align}
Considering the point $-s^n,$ we find
\begin{align}
    -\abar^\top e^n + \oneh (e^n)^\top \hhat e^n&=0, \label{eq:aennHn}
\end{align}
which reduces \eqref{eq:x0eien} to 
\begin{align}
  -(e^i)^\top \hhat e^n =0 \quad \mbox{for }~i \in \{1, 2, \ldots, n-1\}. \label{eq:eiHn}
\end{align}
Thus, $(e^i)^\top \hhat e^n=\hhat_{i,n}=\hhat_{n,i}=0$ for all $i \in \{1 ,2, \dots, n-1\}.$ Combining \eqref{eq:x0h} at $i=n$ and \eqref{eq:aennHn} multiplied by $-1$, we get
\begin{align}
    \abar^\top e^n+ \oneh (e^n)^\top \hhat e^n  = 0 = \abar^\top e^n - \oneh (e^n)^\top \hhat e^n, \label{eq:rearrange}
\end{align}
which implies that $(e^n)^\top \hhat e^n= \hhat_{n,n}=0$. Considering the points $2s^i-s^n$ for $i \in \{1, 2, \dots, n-1\},$ we get
\begin{equation*}
    2 \abar^\top e^i - \abar^\top e^n +2 (e^i)^\top \hhat e^i + \oneh (e^n)^\top \hhat e^n=0.
\end{equation*}
Using \eqref{eq:aennHn}, this simplifies to
\begin{align*}
    2 \abar^\top e^i+ 2 (e^i)^\top \hhat e^i=0.
\end{align*}
By multiplying \eqref{eq:x0h} by 2 and substituting in the above equation, we get $(e^i)^\top \hhat e^i=\hhat_{i,i}=0$ for all $i \in \{1, 2, \dots, n-1\}.$  This now implies $\abar_i = \abar^\top e^i=0$ for all $i \in \{1, 2, \ldots, n\}$, i.e., $\abar=\zero$. Lastly, consider the points $s^i+s^j-s^n$ for $i \neq j, i,j \in \{ 1,2, \dots, n-1\}.$ Since $\hhat_{i,i}=0$ for $i \in \{1, 2, \dots, n\}$,  $\hhat_{i,n}=\hhat_{n,i}=0,$ for $i \in \{1, 2, \dots, n-1\}$ and  $\abar=\zero$, we obtain
\begin{equation}
    (e^i)^\top \hhat e^j=0.
\end{equation}
Thus $\hhat=\zero_{n \times n}$. Therefore, the only solution to \eqref{eq:poisedqisys}  is the trivial solution.

{\bf Case II: $k=0$.}  The proof for this is analogous. 
\end{proof}

Next, we provide an example that serves to show that a set of $(n+1)(n+2)/2$ distinct points in $\R^n$ that is poised for quadratic interpolation is not necessarily a minimal poised set for NSHC.

\begin{ex}\label{ex:nonminNSHC}
Let $x^0=\bbm 0&0 \ebm^\top$ be the point of interest. Consider $\mathcal{X}=$ $\{x^0, e^1, e^2, -e^1, -e^2,$ $ -e^1-e^2\}.$ Then $\mathcal{X}$ is poised for quadratic interpolation, but cannot be expressed as a minimal poised set for NSHC at $x^0$.
\end{ex}

\begin{proof}Using a similar approach as in  Proposition \ref{prop:minimalsetispoisedforqi}, one can verify that $\X$ is poised for quadratic interpolation. Now we show that $\X$ is not a minimal poised set for NSHC at $x^0$ using brute force. We need to build $S=\{ s^1, s^2\}$ such that the matrix corresponding to $S$ is full rank and $x^0+S \subseteq \X$. Hence, the possible choices for $S$ are 
\begin{align}
    S &\in \left \{ \{e^1,e^2\},  \{e^1,-e^2\}, \{e^1,-e^1-e^2\}, \{-e^1,e^2 \}, \right . \notag\\ &\quad\quad\left . \{-e^1-e^2, e^2 \}, \{-e^1,-e^2 \}, \{ -e^1,-e^1-e^2 \},  \{-e^1-e^2, -e^2\}  \right \}. \label{eq:exsets}
\end{align}
{\bf Case I: $S=\{e^1,e^2\}.$} In this case, we need to build $T=\{t^1,t^2\}$ such that the matrix corresponding to $T$ is full rank and 
    $$\{t^1, e^1+t^1,  e^2+t^1, t^2, e^1+t^2,  e^2+t^2\} = \X.$$
We see that the only possible choice of $t^1$ such that $\{t^1, e^1+t^1, e^2+t^1\} \subseteq \X $ is $t^1=-e^1-e^2$ (note $t^1 \neq 0$ as we require full rank). However, the only possible choice of $t^2$ such that $\{t^2, e^1+t^2, e^2+t^2\} \subseteq \X $ is $t^2=-e^1-e^2$. As full rank implies $t^1$ cannot equal $t^2$, we see $S=\{e^1,e^2\}$ cannot provide the desired properties.
\\
{\bf Cases II through VIII:} The other options for $S$ can be eliminated analogously.
\\
Therefore, $\X$ cannot be expressed as a minimal poised set for NSHC at $x^0$.
\end{proof}

Figure \ref{fig:notminset} shows all possible directions connecting two points in the set $\X$ from Example \ref{ex:nonminNSHC}. If $\X$ were a minimal poised set for NSHC at $x^0=\zero$, then it would be possible to choose two directions (lines in blue) emerging from $x^0$ that connect to other points in $\Y$ and these same two directions would be emerging from two other points.
\begin{figure}[ht]
\caption{A set that is poised for QI but not a minimal poised set for NSHC at $x^0$.}\label{fig:notminset}
\includegraphics[scale=0.5]{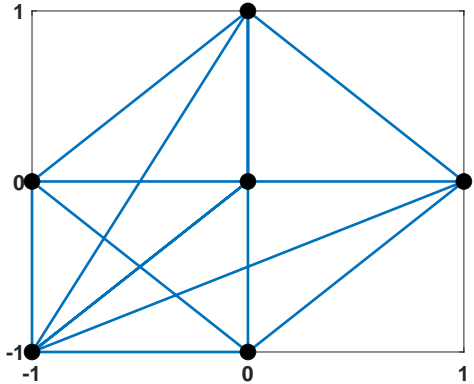}
\centering
\end{figure}

We conclude this section by providing formulae to obtain all the values of the coefficients involved in the quadratic interpolation function of $f$ over a minimal poised set for NSHC of the form $\m(x^0;S,U_k)$, which we denote by  $Q_f(x^0;S,U_k)(x).$ We will see in Section \ref{sec:calculus} that it is valuable to build $Q_f(x^0;S,U_k)(x)$, as it allows for a significant gain in accuracy when defining calculus rules for an approximation of the Hessian. 

\begin{prop} Let $f:\dom f \subseteq \R^n \to \R$.  Let $\m(x^0;S,U_k) \subset \dom f$ be a minimal poised set for NSHC at $x^0$ as constructed in Proposition \ref{prop:reducingfe}. If $k\in\{1, 2, \dots, n\},$ then the Hessian matrix $\h$ of  the quadratic interpolation function $Q_f(x^0;S,U_k)(x)$    is given by $\h=S^{-\top}\hhat S^{-1},$ where the entries of the symmetric matrix $\hhat \in \R^{n \times n}$ are 
\begin{align*}
    \hhat_{i,k}&=-f(x^0+s^i-s^k)+f(x^0+s^i)+f(x^0-s^k)-f(x^0), \quad i \in \{1, 2, \dots, n\}\setminus\{k\},\\
    \hhat_{i,i}&=f(x^0+2s^i-s^k)-2f(x^0+s^i-s^k)+f(x^0-s^k), \quad i \in \{1, 2, \dots, n\}\setminus\{k\},\\
     \hhat_{k,k}&=f(x^0+s^k)+f(x^0-s^k)-2f(x^0),\\
    \hhat_{i,j}&=f(x^0+s^i+s^j-s^k)-f(x^0+s^i-s^k)-f(x^0+s^j-s^k)+f(x^0-s^k),
\end{align*}
for all $i,j \in \{1, 2, \dots, n\}\setminus\{k\}, i \neq j.$ If $k=0,$ then
\begin{align*}
    \hhat_{i,i}&=f(x^0+2s^i)-2f(x^0+s^i)+f(x^0), \quad i \in \{1, 2, \dots, n\},\\
    \hhat_{i,j}&=f(x^0+s^i+s^j)-f(x^0+s^i)-f(x^0+s^j)+f(x^0), \quad i,j \in \{1, 2, \dots, n\}, i \neq j.
\end{align*}
For all $k \in \{0, 1, \dots, n\},$ the vector $\alpha$ associated to $Q_f(x^0;S,U_k)(x)$ is given by  $\alpha=S^{-\top}\abar$, where  the entries of $\abar \in \R^n$ are 
\begin{align*}
    \abar_i&=f(x^0+s^i)-f(x^0)-\frac{1}{2}\hhat_{i,i}-(x^0)^\top\h s^i, \quad i \in \{1, 2,\dots, n\}.
\end{align*}
 The scalar $\alpha_0$ of $Q_f(x^0;S,U_k)(x)$ is 
\begin{align*}
    \alpha_0 &=f(x^0)-\alpha^\top x^0-\frac{1}{2}(x^0)^\top \h x^0.
\end{align*}
\end{prop}
\begin{proof}
The result is obtained by using Definition \ref{def:quadinterpolationfunc}. Let
\begin{align*}
    Q_f(x^0;S,U_k)(x)=\alpha_0+\alpha^\top x+\frac{1}{2}x^\top \h x,
\end{align*}
where $a_0 \in \R, \alpha \in \R^n$ and $\h=\h^\top \in \R^{n \times n}.$ Suppose $k \in \{1, 2, \dots, n\}.$
Evaluating $Q_f(x^0;S,U_k)(x)$ at $x^0$, we obtain
\begin{align}
    \alpha_0+\alpha^\top x^0+\frac{1}{2}(x^0)^\top\h x^0&=f(x^0). \label{eq:x0f}
\end{align}
Evaluating $Q_f(x^0;S,U_k)(x)$ at $x^0+s^i$ and using \eqref{eq:x0f}, we obtain
\begin{align}
    f(x^0)+\abar^\top e^i+(x^0)^\top \hba e^i+\frac{1}{2}\hhat_{i,i}&=f(x^0+s^i), \quad i \in \{1, \dots, n\}, \label{eq:x0psi}
\end{align}
where $\abar^\top=\alpha^\top S, \hba=\h S$ and $\hhat=S^\top \h S.$ Evaluating $Q_f(x^0;S,U_k)(x)$ at $x^0+s^i-s^k$ and using \eqref{eq:x0f} and \eqref{eq:x0psi}, we obtain
\begin{equation}
   \scalemath{1}{ f(x^0+s^i)-\abar^\top s^k-(x^0)^\top \hba e^k-\hhat_{i,k}+\frac{1}{2}\hhat_{k,k}=f(x^0+s^i-s^k),  i \in \{1, \dots, n\}\setminus \{k\}.}\label{eq:x0simsk}
\end{equation}
Evaluating $Q_f(x^0;S,U_k)(x)$ at $x^0-s^k$ and using \eqref{eq:x0f}, we find
\begin{align}
    -\abar^\top e^k-(x^0)^\top \hba e^k+ \frac{1}{2} \hhat_{k,k}&=f(x^0-s^k)-f(x^0). \label{eq:x0msk}
\end{align}
Substituting \eqref{eq:x0msk} into \eqref{eq:x0simsk}, we obtain
\begin{equation}
    \hhat_{i,k}=\hhat_{k,i}=-f(x^0+s^i-s^k)+f(x^0+s^i)+f(x^0-s^k)-f(x^0),  i \in \{1, \dots, n\}\setminus \{k\}. \label{eq:hhatik}
\end{equation}
Evaluating $Q_f(x^0;S,U_k)(x)$ at $x^0+2s^i-s^k$ and using \eqref{eq:x0f}, \eqref{eq:x0psi} and \eqref{eq:x0msk}, we find
\begin{align}
     \scalemath{0.9}{2f(x^0+s^i)-2f(x^0)+f(x^0-s^k)+\hhat_{i,i}-2\hhat_{i,k}=f(x^0+2s^i-s^k) \label{eq:x0p2simsk}, i \in \{1, \dots, n\}\setminus\{k\}.}
\end{align}
Substituting \eqref{eq:hhatik} in \eqref{eq:x0p2simsk}, we find
\begin{align*}
    \hhat_{i,i}&=f(x^0+2s^i-s^k)+f(x^0-s^k)-2f(x^0+s^i-s^k), \quad i \in \{1, 2, \dots, n\}\setminus\{k\}.
\end{align*}
Evaluating $Q_f(x^0;S,U_k)(x)$ at $x^0+s^i+s^j-s^k$ and using \eqref{eq:x0f}, \eqref{eq:x0psi} and \eqref{eq:x0msk}, we find
\begin{align*}
    \hhat_{i,j}=\hhat_{j,i}&=f(x^0+s^i+s^j-s^k)-f(x^0+s^i-s^k)-f(x^0+s^j-s^k)+f(x^0-s^k), 
\end{align*}
for $i,j \in \{1, 2, \dots, n\}\setminus\{k\}, i \neq j.$ Rearranging \eqref{eq:x0msk}, we get
\begin{align}
    \frac{1}{2} \hhat_{k,k}-f(x^0-s^k)+f(x^0)&=\abar^\top e^k+(x^0)^\top\hba e^k \label{eq:rearrangex0msk}.
\end{align}
Substituting \eqref{eq:rearrangex0msk} into \eqref{eq:x0psi} for $i=k$, we obtain
\begin{align*}
    \hhat_{k,k}&=f(x^0+s^k)+f(x^0-s^k)-2f(x^0).
\end{align*}
The entries of the vector $\abar$ are found by isolating $\abar^\top e^i$ in \eqref{eq:x0psi}. We obtain
\begin{align*}
    \abar_i&=f(x^0+s^i)-f(x^0)-\frac{1}{2} \hhat_{i,i}-(x^0)^\top\h s^i, \quad i \in \{1, 2, \dots, n\} 
\end{align*}
where $\alpha_i=S^{-\top}\abar_i.$ Lastly, the scalar $\alpha_0$ is obtained from \eqref{eq:x0f}. We find
\begin{align*}
    \alpha_0&=f(x^0)-\alpha^\top x^0-\frac{1}{2}(x^0)^\top \h x^0.
\end{align*}
If $k=0$, a similar process can be applied to obtain $\h, \alpha$ and $\alpha_0.$ \qedhere 
\end{proof}

It is worth emphasizing that $Q_f(x^0;S,U_k)(x)$ can be obtained for free in terms of function evaluations whenever  $\nabla^2_s f(x^0;S,U_k)$ has already been computed. Indeed, all the coeficients of $Q_f(x^0;S,T)$ are computed using the same function evaluations used in the computation of $\nabla_s^2 f(x^0;S,U_k).$ 

\section{Error bound for the nested-set Hessian}\label{sec:error}

In this section, an error bound is defined for the nested-set Hessian. This error bound will demonstrate that, in the language of Definition \ref{df:orderNaccuracy}, the nested-set Hessian is an order-1 Hessian accuracy approximation technique. Thereafter, a simplified formula for the error bound  is provided in the particular case where the set used to compute the nested-set Hessian is the $k^{\mbox{th}}$-canonical minimal poised set for NSHC at $x^0$ scaled by a nonzero scalar $\beta$.

\begin{prop}[Error bound for the nested-set Hessian] \label{prop:ebsimplexhessian}
Let $f:\dom f \subseteq \R^n\to\R$ be $\mathcal{C}^{3}$ on $B(x^0;\overline{\Delta})$ where $x^0 \in \dom f$ is the point of interest and $\overline{\Delta}>0$. Denote by $L_{\nabla^2 f}$ the Lipschitz constant of $\nabla^2 f$ on $B(x^0,\overline{\Delta}).$ Let $S=\bbm s^1&s^2&\cdots&s^m\ebm \in \R^{n \times m}$ and $T=\bbm t^1&t^2&\cdots&t^k\ebm \in \R^{n \times k}$ be full row rank  and $B(x^0+s^i;\Delta_T)\subset B(x^0;\overline{\Delta})$  for all $i$. 
 Then
\begin{align}\label{eq:ebgsh}
\|\sh f(x^0;S,T)-\nabla^2 f(x^0)\| &\leq \frac{ m \sqrt{k}}{3}L_{\nabla^2 f}  \left ( 2\frac{\Delta_u}{\Delta_l}+3 \right ) \left \Vert (\widehat{S}^\top) ^\dagger\right \Vert \left \Vert \widehat{T}^\dagger \right \Vert \Delta_u,
\end{align}
where $\widehat{S}=S/\Delta_S, \widehat{T}=T/\Delta_T, \Delta_u=\max \{\Delta_S, \Delta_T\},$ and $\Delta_l=\min \{\Delta_S, \Delta_T\}.$ Moreover, if $\Delta_u=\Delta_l,$ then
\begin{align*}
    \|\sh f(x^0;S,T)-\nabla^2 f(x^0)\| &\leq \frac{ 5m \sqrt{k} }{3}  L_{\nabla^2 f} \left \Vert (\widehat{S}^\top) ^\dagger\right \Vert \left \Vert \widehat{T}^\dagger \right \Vert \Delta_u.
\end{align*}
\end{prop}
\begin{proof} As $S$ is full  row rank, we have
\begin{align}
   \Vert \sh f(x^0;S,T)-\nabla^2 f(x^0)\Vert&=\Vert  (S^\top)^\dagger \dnsf(x^0;S,T)-(S^\top)^\dagger S^\top \nabla^2f(x^0) \Vert \notag\\
    &\leq \Vert(S^\top)^\dagger \Vert \Vert \dnsf(x^0;S,T)-S^\top \nabla^2 f(x^0)  \Vert \notag\\
    &= \frac{1}{\Delta_S}\Vert(\widehat{S}^\top)^\dagger \Vert \Vert \dnsf(x^0;S,T)-S^\top \nabla^2 f(x^0)  \Vert. \label{eq:youngster}
\end{align}
Now we find a bound for  $\Vert \dnsf(x^0;S,T)-S^\top \nabla^2 f(x^0)  \Vert.$ We have
\begin{align}
    &\Vert \dnsf(x^0;S,T)-S^\top \nabla^2 f(x^0)  \Vert \notag\\
    \leq& \sum_{i=1}^m \left \Vert \left ( \ns f(x^0+s^i;T)-\ns f(x^0;T)\right )^\top -(s^i)^\top\nabla^2 f(x^0) \right \Vert \notag\\
    =&\sum_{i=1}^m \left \Vert \left ((T^\top)^\dagger\left ( \delta_f (x^0+s^i;T) - \delta_f(x^0;T)\right ) \right )^\top -(s^i)^\top \nabla^2 f(x^0) \right \Vert \notag\\
    =&\sum_{i=1}^m \left \Vert \left ( \delta_f (x^0+s^i;T)-\delta_f(x^0;T)\right )^\top T^\dagger -(s^i)^\top \nabla^2 f(x^0)  T T^\dagger \right \Vert \notag \\
    \leq& \left \Vert  T^\dagger \right \Vert \sum_{i=1}^m \left \Vert \left ( \delta_f (x^0+s^i;T)-\delta_f(x^0;T)\right )^\top  -(s^i)^\top \nabla^2 f(x^0)T \right \Vert \notag\\
     =& \left \Vert  T^\dagger \right \Vert \sum_{i=1}^m \left (\sum_{j=1}^k \left \vert f(x^0+s^i+t^j)-f(x^0+s^i)-f(x^0+t^j)+f(x^0)-(s^i)^\top \nabla^2f(x^0)t^j \right \vert^2 \right )^{\frac{1}{2}}. \label{eq:doublesum}
\end{align}
Using Taylor's Theorem, we know
\begin{equation}
   \scalemath{0.97}{ f(x^0+s^i+t^j)=f(x^0)+\nabla f(x^0)^\top (s^i+t^j)+\frac{1}{2} (s^i+t^j)^\top \nabla^2 f(x^0) (s^i+t^j)+ R_2(x^0+s^i+t^j),} \label{eq:fx0hidj}
\end{equation}
where $R_2(x^0+s^i+t^j)$ is the remainder term (see, for example, \cite{Burden2016}). By Taylor's theorem, we can also write
\begin{align}
    f(x^0+s^i)&=f(x^0)+\nabla f(x^0)^\top (s^i)+\frac{1}{2} (s^i)^\top \nabla^2 f(x^0) s^i+ R_2(x^0+s^i), \label{eq:fx0hi}\\
    f(x^0+t^j)&=f(x^0)+\nabla f(x^0)^\top t^j+\frac{1}{2} (t^j)^\top \nabla^2 f(x^0) t^j+ R_2(x^0+t^j) \label{eq:fx0dj}.
\end{align}
Subtracting \eqref{eq:fx0hi} and \eqref{eq:fx0dj} from \eqref{eq:fx0hidj}, we obtain
\begin{align*}
    &f(x^0+s^i+t^j)-f(x^0+s^i)-f(x^0+t^j)\\
    =&-f(x^0)+(s^i)^\top \nabla^2 f(x^0)t^j+R_2(x^0+s^i+t^j)-R_2(x^0+s^i)-R_2(x^0+t^j).
\end{align*}
Rearranging and taking the norm on both sides, we get
\begin{align}
     &\vert f(x^0+s^i+t^j)-f(x^0+s^i)-f(x^0+t^j) +f(x^0)-(s^i)^\top \nabla^2 f(x^0)t^j \vert \notag\\
     \leq& \Vert R_2(x^0+s^i+t^j)\Vert +\Vert R_2(x^0+s^i)\Vert + \Vert R_2(x^0+t^j) \Vert \notag \\
     \leq& \frac{1}{6}L_{\nabla^2 f} \Vert s^i+t^j\Vert^3+\frac{1}{6}L_{\nabla^2 f} \Vert s^i\Vert^3+\frac{1}{6}L_{\nabla^2 f} \Vert t^j\Vert^3 \notag\\
     \leq& \frac{1}{6} L_{\nabla^2 f} (\Delta_S+\Delta_T)^3+\frac{1}{6} L_{\nabla^2 f} \Delta_S^3+\frac{1}{6} L_{\nabla^2 f} \Delta_T^3 \notag \\
     =&\frac{1}{6}L_{\nabla^2 f} \left ( (\Delta_S+\Delta_T)^3+\Delta_S^3+\Delta_T^3\right ). \label{eq:Deltapower3}
\end{align}
Using \eqref{eq:Deltapower3} in \eqref{eq:doublesum}, we get the inequality
\begin{align}
     \Vert \dnsf(x^0;S,T)-S^\top \nabla^2 f(x^0)  \Vert \leq \left \Vert  T^\dagger \right \Vert \frac{m \sqrt{k}}{6}  L_{\nabla^2 f} \left ( (\Delta_S+\Delta_T)^3+\Delta_S^3+\Delta_T^3 \right ). \label{eq:boundDelta}
\end{align}
Using \eqref{eq:boundDelta} in \eqref{eq:youngster}, we now have
\begin{align*}
    \Vert \sh f(x^0;S,T)-\nabla^2 f(x^0)\Vert &\leq \frac{m \sqrt{k}}{6}  L_{\nabla^2 f}   \left ( \frac{(\Delta_S+\Delta_T)^3+\Delta_S^3+\Delta_T^3}{\Delta_S \Delta_T} \right ) \Vert(\widehat{S}^\top)^\dagger \Vert  \left \Vert  \widehat{T}^\dagger \right \Vert\\
    &=\frac{m \sqrt{k}}{6}  L_{\nabla^2 f} \left ( \frac{2\Delta_S^3+2\Delta_T^3}{\Delta_S \Delta_T}+3\Delta_S+3\Delta_T \right ) \Vert(\widehat{S}^\top)^\dagger \Vert  \left \Vert  \widehat{T}^\dagger \right \Vert.
\end{align*}
Applying $\Delta_u=\max \{ \Delta_S, \Delta_T\}$ and $\Delta_l= \min \{ \Delta_S, \Delta_T\}$, we get
\begin{align*}
    \Vert \sh f(x^0;S,T)-\nabla^2 f(x^0)\Vert &\leq \frac{m \sqrt{k} }{6}  L_{\nabla^2 f} \left ( \frac{4\Delta_u^2}{\Delta_l}+6\Delta_u \right ) \Vert(\widehat{S}^\top)^\dagger \Vert  \left \Vert  \widehat{T}^\dagger \right \Vert\\
    &=\frac{m \sqrt{k} }{3}  L_{\nabla^2 f} \left ( 2\frac{\Delta_u}{\Delta_l}+3 \right ) \Vert(\widehat{S}^\top)^\dagger \Vert  \left \Vert  \widehat{T}^\dagger \right \Vert \Delta_u.
\end{align*}
Finally, in the case where $\Delta_u=\Delta_l$, this reduces to 
\begin{align*}
    \Vert \sh f(x^0;S,T)-\nabla^2 f(x^0)\Vert &\leq \frac{5 m \sqrt{k}}{3}  L_{\nabla^2 f} \left \Vert(\widehat{S}^\top)^\dagger  \right \Vert  \left \Vert  \widehat{T}^\dagger \right \Vert \Delta_S.
\end{align*}
\end{proof}

Note that if $\Delta_u \neq \Delta_l$, when $\Delta_u$ tends to zero, we need $\frac{\Delta_u}{\Delta_l}$ to be finite, so that the error bound goes to zero. One way of ensuring this is to have $\Delta_u$ and $\Delta_l$ decrease at the same rate. This is certainly the case when a minimal poised set for NSHC of the form $\m(x^0;S,U_k)$ is considered, as $U_k$ is built from $S$. The next corollary provides a simplified bound when the $k^{\mbox{th}}$-minimal poised set for NSHC scaled by $\beta\neq0$ is used to compute the nested-set Hessian.

\begin{cor}\label{cor:EBminsetforSHC}
Let $f:\dom f \subseteq \R^n\to\R$ be $\mathcal{C}^{3}$ on $B(x^0,\overline{\Delta})$ and denote by $L_{\nabla^2 f}$ the Lipschitz constant of $\nabla^2 f$ on $B(x^0,\overline{\Delta}).$ Let $\m(x^0;\beta \Id,\beta E_k)\subset B(x^0,\overline{\Delta})$ be a  minimal poised set for NSHC at $x^0$, where $\beta$ is a nonzero scalar.
Then, for $k=0$
 \begin{align*}
\|\sh f(x^0;\beta \Id, \beta E_k)-\nabla^2 f(x^0)\| &\leq \frac{5}{3} \, n\sqrt{n} \, L_{\nabla^2 f} \beta.
\end{align*}
and for $k\in \{1, 2, \ldots, n\}$
\begin{align*}
\|\sh f(x^0;\beta \Id, \beta E_k)-\nabla^2 f(x^0)\| &\leq \frac{11}{2} \, n^2 \, L_{\nabla^2 f} \beta.
\end{align*}
\end{cor}

\begin{proof}
 From  Proposition \ref{prop:ebsimplexhessian}, we know
\begin{align}
\|\sh f(x^0;\beta \Id, \beta E_k)-\nabla^2 f(x^0)\| &\leq \frac{ n \sqrt{n}}{3}L_{\nabla^2 f}  \left ( 2\frac{\Delta_u}{\Delta_l}+3 \right ) \left \Vert (\widehat{\beta \Id}^\top) ^{-1}\right \Vert \left \Vert \widehat{\beta E_k}^{-1} \right \Vert \Delta_u, \label{eq:eb}
\end{align}
where $\widehat{\beta \Id}=\beta \Id/\Delta_{\beta \Id}, \widehat{\beta E_k}= \beta E_k/\Delta_{\beta E_k}, \Delta_u=\max \{\Delta_{\beta \Id}, \Delta_{\beta E_k}\},$ $\Delta_l=\min \{\Delta_{\beta \Id}, \Delta_{\beta E_k}\}.$ 

If $k=0$, then $\Delta_{\beta \Id}=\Delta_{\beta E_0}=\beta$ and $\left \Vert (\widehat{\beta \Id}^\top) ^{-1}\right \Vert  = \left \Vert \widehat{\beta E_0}^{-1} \right \Vert = 1$.  The result follows trivially.

If $k\in \{1, 2, \ldots, n\}$, then we have $\Delta_{\beta \Id}=\beta$ and $\Delta_{\beta E_k}=\sqrt{2} \beta$.  Without loss of generality, let $k=n$.  Note that 
$$\widehat{\beta \Id}=\Id ~~\mbox{and}~~ \widehat{\beta E_n}=\frac{1}{\sqrt{2}}\bbm \Id_{n-1} &\zero_{n-1 \times 1}\\ -\one_{1 \times n-1}&-1\ebm.$$
Using Corollary 3.2 in \cite{Lu2002}, we get that $\widehat{\beta E_n}^{-1}=2 \widehat{\beta E_n}.$ Substituting all  these specific values in \eqref{eq:eb}, we obtain
\begin{align*}
    \|\sh f(x^0;\beta \Id,\beta E_k)-\nabla^2 f(x^0)\| &\leq \frac{n^{\frac{3}{2}}}{3} L_{\nabla^2 f} \left (2\sqrt{2} +3 \right ) \left \Vert \widehat{\beta E_n}^{-1} \right \Vert \sqrt{2} \, \beta \\
    &\leq \frac{\sqrt{2} n^{\frac{3}{2}}}{3} \left ( 2\sqrt{2}+3\right ) L_{\nabla^2 f} \left \Vert \widehat{\beta E_n}^{-1} \right \Vert_F \beta \\
    &=\frac{2 n^{\frac{3}{2}}}{3} \left ( 2\sqrt{2}+3\right ) L_{\nabla^2 f} \sqrt{2n-1} \, \beta \\
    &<\left ( \frac{8+6\sqrt{2}}{3}\right )\, n^2 \, L_{\nabla^2 f} \beta\\
    &< \frac{11}{2}n^2 \, L_{\nabla^2 f} \beta.\qedhere
\end{align*}
\end{proof}

Note that the error bounds defined in Proposition \ref{prop:ebsimplexhessian} and Corollary \ref{cor:EBminsetforSHC} are factors of the Lipschitz constant $L_{\nabla^2 f}$. Hence, if $f$ is a quadratic function, then $L_{\nabla^2 f}=0$ and the error bounds are equal to zero. Therefore, if the assumptions of Proposition \ref{prop:ebsimplexhessian} hold and $f$ is a quadratic function, then  $\nabla^2_s f(x^0;S,T)$ is perfectly accurate.

\section{ The simplex calculus Hessian and the quadratic calculus Hessian}\label{sec:calculus}

This section contains the development of two approximation techniques of the Hessian, based on the calculus rules of the true Hessian. These novel approximation techniques are named simplex calculus Hessian and quadratic calculus Hessian. A product rule, a quotient rule and a power rule of functions are advanced. The simplex calculus Hessian and quadratic calculus Hessian provide alternative techniques to approximate the Hessian whenever more than one function is involved. These approximation techniques make it possible to consider the functions involved separately, rather than considering them all as one function, as this is the case for the nested-set Hessian. Once the rules are defined,  an error bound is provided for each rule. The error bounds developed show that every Hessian approximation technique in this section are order-1 Hessian accuracy  approximation depending on $\Delta_u,$ the maximal radius between the two sets of directions $S$ and $T$.  

All results use Proposition \ref{prop:ebsimplexhessian}.  In addition, we use the following bounds on gradient approximation
\begin{align*}
\|\nabla Q_f(x^0;S,T)(x^0)-\nabla f(x^0)\|&\leq\frac{3(1+\sqrt2)\sqrt p}{2}L_{\nabla^2f}\left\|\widehat{Q}^{-1}\right\|(\Delta_S+\Delta_T)^2,\qquad\cite[\mbox{Theorem 3}]{conn2008geometry}\nonumber\\
&\leq\frac{3(1+\sqrt2)\sqrt p}{2}L_{\nabla^2f}\left\|\widehat{Q}^{-1}\right\|(2\Delta_u)^2,
\end{align*}
and
\begin{align*}
    \Vert \nabla_s f(x^0;T)-\nabla f(x^0) \Vert &\leq \frac{\sqrt{k}}{2}L_{\nabla f} \Vert (T^\top)^\dagger \Vert \Delta_T,
    \qquad \mbox{[Proposition ~\ref{prop:GSGerror}]}\nonumber\\
    & \leq \frac{\sqrt{k}}{2}L_{\nabla f} \Vert (T^\top)^\dagger \Vert \Delta_u. 
    \end{align*}
These inspire the definitions
    $$E_{\nabla Q_f}:=6(1+\sqrt2)\sqrt pL_{\nabla^2f}\left\|\widehat{Q}^{-1}\right\|, \quad
    E_{\nabla_s f}:=\frac{\sqrt{k}}{2}L_{\nabla f} \Vert (T^\top)^\dagger \Vert,$$
and 
$$ E_{\nabla^2_s f}=\frac{ m \sqrt{k}}{3}L_{\nabla^2 f}  \left ( 2\frac{\Delta_u}{\Delta_l}+3 \right ) \left \Vert (\widehat{S}^\top) ^\dagger\right \Vert \left \Vert \widehat{T}^\dagger \right \Vert.$$
Observe that all of the error bounds used have a factor of at least $\Delta_u$, which allows control over the error bounds.

Before introducing the simplex calculus Hessian, Table \ref{ruletable} presents the calculus rules for the true Hessians of the types of functions that are of interest.\bigskip\\
\begin{table}[H]
\scalebox{0.97}{
\begin{tabular}{|l|l|}\hline
\textbf{Calculus rule}&\textbf{Equation}\\\hline
Product Rule&$\nabla^2(fg)=(\nabla^2f)g+\nabla f(\nabla g)^\top+\nabla g(\nabla f)^\top+(\nabla^2g)f$\\\hline
Quotient Rule&$\nabla^2\left(\frac{f}{g}\right)=\frac{1}{g^3}[g^2(\nabla^2f-f\nabla^2g)+2f\nabla g(\nabla g)^\top-g(\nabla f(\nabla g)^\top+\nabla g(\nabla f)^\top)]$\\\hline
Power Rule&$\nabla^2(f^p)=pf^{p-1}\nabla^2f+p(p-1)f^{p-2}\nabla f(\nabla f)^\top$\\\hline
\end{tabular}}\caption{Calculus rules for the Hessian.}\label{ruletable}
\end{table} 

We define the simplex calculus Hessians in Table \ref{table:ruletable2} to be analogous to the true Hessians and use the notation $\nabla^2_{sc}(\cdot)(x^0;S,T)$.
\begin{table}[H] 
\scalebox{0.92}{
\begin{tabular}{|l|l|}\hline
\textbf{Calculus rule}&\textbf{Equation}\\\hline
Product Rule&$\nabla^2_{sc} (fg)=(\nabla^2_{s} f)g+\nabla_{s}f(\nabla_sg)^\top+\nabla_sg(\nabla_sf)^\top+(\nabla^2_sg)f$\\\hline
Quotient Rule&$\nabla^2_{sc}\left(\frac{f}{g}\right)=\frac{1}{g^3}[g^2(\nabla^2_sf-f\nabla^2_sg)+2f\nabla_sg(\nabla_sg)^\top-g(\nabla_sf(\nabla_sg)^\top+\nabla_sg(\nabla_sf)^\top)]$\\\hline
Power Rule&$\nabla^2_{sc}(f^p)=pf^{p-1}\nabla^2_sf+p(p-1)f^{p-2}\nabla_sf(\nabla_sf)^\top$\\\hline
 \end{tabular}}\caption{Calculus rules for the simplex calculus Hessian.} \label{table:ruletable2}\end{table}

An alternative  approximation technique can be defined by replacing all the generalized simplex gradients $\nabla_s f(x^0;T)$ and $\nabla_s g(x^0;T)$ in Table \ref{table:ruletable2} by $\nabla Q_f(x^0;S,T)$ and $\nabla Q_g(x^0;S,T)$, respectively. This gives rise to the quadratic calculus Hessian, denoted $\nabla^2_{qc}(\cdot)(x^0;S,T).$  Naturally, the quadratic calculus Hessian requires the construction of the quadratic interpolations, for instance $Q_f (x^0;S,T)$ and $Q_g(x^0;S,T)$ in the product rule. Notice that using  $\nabla Q_f(x^0;S,T)$ and $\nabla Q_g(x^0;S,T)$ instead of $\nabla_s f(x^0;T)$ and $\nabla_s g(x^0;T) $ in the formulae presented in Table \ref{table:ruletable2}  does not require more function evaluations. Indeed, building $Q_f(x^0;S,T)$ reuses the  same function evaluations involved in the computation of $\nabla^2_s f(x^0;S,T)$ and similarly for $Q_g(x^0;S,T)$. Therefore, the computational cost between these two Hessian approximation techniques in terms of function evaluations is equal. The  formulae for the quadratic calculus Hessian are presented in Table \ref{table:ruletable3}.
\begin{table}[H] 
\scalebox{0.88}{
\begin{tabular}{|l|l|}\hline
\textbf{Calculus rule}&\textbf{Equation}\\\hline
Product Rule&$\nabla^2_{qc} (fg)=(\nabla^2_{s} f)g+\nabla Q_f(\nabla Q_g)^\top+\nabla Q_g(\nabla Q_f)^\top+(\nabla^2_sg)f$\\\hline
Quotient Rule&$\nabla^2_{qc}\left(\frac{f}{g}\right)=\frac{1}{g^3}[g^2(\nabla^2_sf-f\nabla^2_sg)+2f\nabla Q_g(\nabla Q_g)^\top-g(\nabla Q_f(\nabla Q_g)^\top+\nabla Q_g(\nabla Q_f)^\top)]$\\\hline
Power Rule&$\nabla^2_{qc}(f^p)=pf^{p-1}\nabla^2_sf+p(p-1)f^{p-2}\nabla Q_f(\nabla Q_f)^\top$\\\hline
\end{tabular}}\caption{Calculus rules for the quadratic calculus Hessian.} \label{table:ruletable3}\end{table}

We now present error bounds for the product rule of the simplex calculus Hessian and the quadratic calculus Hessian as a proposition with full proof.  We list the remaining error bounds as corollaries; the proofs of all rules follow the same method.

\begin{prop}[Product Rule error bound]\label{prop:productrule}Let $f,g:\R^n\to\R$ be $\mathcal{C}^3$ on $B(x^0;\overline{\Delta})$, where $x^0\in\dom f\cap\dom g$ is the point of interest. Let $S=[s^1~~s^2~~\cdots~~s^m]\in\R^{n\times m}$ and $T=[t^1~~t^2~~\cdots~~t^k]\in\R^{n\times k}$ be full row rank. Let $B(x^0+s^i;\Delta_T)\subset B(x^0;\overline{\Delta})$ for all $i\in\{1,2,\ldots,m\}$. Then
\begin{equation}\label{producterror}
\|\nabla^2_{sc}(fg)(x^0;S,T)-\nabla^2(fg)(x^0)\|\leq(E_{\nabla^2_s f}|g(x^0)|+E_{\nabla^2_sg}|f(x^0)|+2M^s_{fg})\Delta_u,
\end{equation}
where\begin{equation*}
M^s_{fg}=\min\left\{\begin{array}{l}E_{\nabla_s f}E_{\nabla_s g}\Delta_u+E_{\nabla_s g}\|\nabla f(x^0)\|+E_{\nabla_s f}\|\nabla g(x^0)\|,\\E_{\nabla_s f}\|\nabla_s g(x^0;T)\|+E_{\nabla_s g}\|\nabla f(x^0)\|,\\E_{\nabla_s g}\|\nabla_s f(x^0;T)\|+E_{\nabla_s f}\|\nabla g(x^0)\|\end{array}\right\}.
\end{equation*}
Moreover, if $\s(x^0;S,T)$ is poised for quadratic interpolation, then
\begin{equation}
\label{eq:qchproducterror}\|\nabla^2_{qc}(fg)(x^0;S,T)-\nabla^2(fg)(x^0)\|
    \leq\left(E_{\nabla^2_s f}|g(x^0)|+E_{\nabla^2_sg}|f(x^0)|+2M^q_{fg}\Delta_u \right)\Delta_u,
\end{equation}
where
\begin{equation}\label{eq:Mqfg}
M^q_{fg}=\min\left\{\begin{array}{l}
    E_{\nabla Q_f}E_{\nabla Q_g}\Delta_u^2+E_{\nabla Q_g}\|\nabla f(x^0)\|+E_{\nabla Q_f}\|\nabla g(x^0)\|,\\
    E_{\nabla Q_f}\|\nabla Q_g(x^0;S,T)(x^0)\|+E_{\nabla Q_g}\|\nabla f(x^0)\|,\\
    E_{\nabla Q_g}\|\nabla Q_f(x^0;S,T)(x^0)\|+E_{\nabla Q_f}\|\nabla g(x^0)\|\end{array}\right\}.
\end{equation}
\end{prop}
\begin{proof}The error bounds for both Hessian approximation techniques are obtained using an identical process. For this reason, we prove only \eqref{eq:qchproducterror} and \eqref{eq:Mqfg}. To save space in the calculations that follow, we use the shorthand $\nabla^2f(x^0)=\nabla^2f$, $\nabla_s^2f(x^0;S,T)=\nabla^2_sf,$ $\nabla Q_f(x^0;S,T)(x^0)= \nabla Q_f$ and similar, as $x^0,S$ and $T$ are fixed and there is no ambiguity. By using the triangle inequality of norms and substituting the product rules from Tables \ref{ruletable} and \ref{table:ruletable3} into the norm on the left-hand side of \eqref{producterror}, we see that it is bounded as follows:
\begin{align}
&\|\nabla^2_s(fg)-\nabla^2(fg)\|\nonumber\\
=&\scalemath{0.91}{\|g(x^0)(\nabla^2_sf-\nabla^2f)+f(x^0)(\nabla^2_sg-\nabla^2g)+\nabla Q_f(\nabla Q_g)^\top-\nabla f(\nabla g)^\top+\nabla Q_g(\nabla Q_f)^\top-\nabla g(\nabla f)^\top\|}\nonumber\\
\leq&\scalemath{0.97}{|g(x^0)|E_{\nabla^2_sf}\Delta_u+|f(x^0)|E_{\nabla^2_sg}\Delta_u+\|\nabla Q_f(\nabla Q_g)^\top-\nabla f(\nabla g)^\top\|+\|\nabla Q_g(\nabla Q_f)^\top-\nabla g(\nabla f)^\top\|}\nonumber\\
=&(|g(x^0)|E_{\nabla^2_sf}+|f(x^0)|E_{\nabla^2_sg})\Delta_u+2\|\nabla Q_f(\nabla Q_g)^\top-\nabla f(\nabla g)^\top\|.\label{ineq2}
\end{align}
Now we want to eliminate the term $2\|\nabla Q_f(\nabla Q_g)^\top-\nabla f(\nabla g)^\top\|$ from \eqref{ineq2}. The first method makes use of the product $(\nabla Q_f-\nabla f)(\nabla Q_g-\nabla g)^\top$:
\begin{align*}
(\nabla Q_f-\nabla f)(\nabla Q_g-\nabla g)^\top&=\nabla Q_f(\nabla Q_g)^\top-\nabla f(\nabla Q_g)^\top-\nabla Q_f(\nabla g)^\top+\nabla f(\nabla g)^\top\\
&=\scalemath{0.91}{\nabla Q_f(\nabla Q_g)^\top-\nabla f(\nabla g)^\top+2\nabla f(\nabla g)^\top-\nabla f(\nabla Q_g)^\top-\nabla Q_f(\nabla g)^\top.}
\end{align*}
Hence, 
\begin{align*}
&\|\nabla Q_f(\nabla Q_g)^\top-\nabla f(\nabla g)^\top\|\\
&\quad=\|(\nabla Q_f-\nabla f)(\nabla Q_g-\nabla g)^\top-2\nabla f(\nabla g)^\top+\nabla f(\nabla Q_g)^\top+\nabla Q_f(\nabla g)^\top\|\\
&\quad=\|(\nabla Q_f-\nabla f)(\nabla Q_g-\nabla g)^\top+\nabla f(\nabla Q_g-\nabla g)^\top+(\nabla Q_f-\nabla f)(\nabla g)^\top\|\\
&\quad\leq E_{\nabla Q_f}E_{\nabla Q_g}\Delta_u^4+\|\nabla f\|E_{\nabla Q_g}\Delta_u^2+\|\nabla g\|E_{\nabla Q_f}\Delta_u^2\\
&\quad=(E_{\nabla Q_f}E_{\nabla Q_g}\Delta_u^3+\|\nabla f\|E_{\nabla Q_g}\Delta_u+\|\nabla g\|E_{\nabla Q_f}\Delta_u)\Delta_u.
\end{align*}
Thus, the right-hand side of \eqref{ineq2} becomes
\begin{equation*}
\left [|g(x^0)|E_{\nabla^2_s f}+|f(x^0)|E_{\nabla^2_s g}+2(E_{\nabla Q_f}E_{\nabla Q_g}\Delta_u^2+\|\nabla f\|E_{\nabla Q_g}+\|\nabla g\|E_{\nabla Q_f})\Delta_u \right ]\Delta_u,\end{equation*}whose expression in parentheses is the first expression in \eqref{eq:Mqfg}. The second method of eliminating $2\|\nabla Q_f(\nabla Q_g)^\top-\nabla f(\nabla g)^\top\|$ from \eqref{ineq2} is to add and subtract $\nabla f(\nabla Q_g)^\top$ to obtain
\begin{equation}\label{ineq4}
\nabla Q_f(\nabla Q_g)^\top-\nabla f(\nabla g)^\top=(\nabla Q_f-\nabla f)(\nabla Q_g)^\top+\nabla f(\nabla Q_g-\nabla g)^\top.
\end{equation}
Substituting \eqref{ineq4} into \eqref{ineq2} and applying the triangle inequality yields another right-hand side of \eqref{ineq2}:
\begin{equation*}
[|g(x^0)|E_{\nabla^2_s f}+|f(x^0)|E_{\nabla^2_s g}+2(\|\nabla Q_g\|E_{\nabla Q_f}+\|\nabla f\|E_{\nabla Q_g})\Delta_u]\Delta_u,
\end{equation*}
whose expression in parentheses is the second expression in \eqref{eq:Mqfg}. Finally, adding and subtracting $\nabla Q_f(\nabla g)^\top$ to \eqref{ineq2} gives
\begin{equation*}
\nabla Q_f(\nabla Q_g)^\top-\nabla f(\nabla g)^\top=\nabla Q_f(\nabla Q_g-\nabla g)^\top+(\nabla Q_f-\nabla f)(\nabla g)^\top,
\end{equation*}
which we substitute into the right-hand side of \eqref{ineq2} to find
\begin{equation*}
[|g(x^0)|E_{\nabla^2_sf}+|f(x^0)|E_{\nabla^2_s g}+2(\|\nabla Q_f\|E_{\nabla Q_g}+\|\nabla g\|E_{\nabla Q_f})\Delta_u]\Delta_u,
\end{equation*}
whose expression in parentheses is the third expression in \eqref{eq:Mqfg}. Therefore, \eqref{eq:qchproducterror} holds.
\end{proof}
\begin{note}
Proposition \ref{prop:productrule} can be extended to the product of any finite number of functions by induction: $f_1f_2f_3:=fg$ with $f=f_1$ and $g=f_2f_3$, etc. We leave this for future consideration.
\end{note}
Analyzing the error bound in \eqref{producterror}, we see that if $f$ and $g$ are linear functions, then the error bound is equal to 0. Therefore, the  simplex calculus Hessian $\nabla_{sc}^2(fg)(x^0;S,T)$ is perfectly accurate in this case. Analyzing the error bound in \eqref{eq:qchproducterror}, we see that if $f$ and $g$ are quadratic functions, then the error bound is equal to zero. Hence, the quadratic calculus Hessian $\nabla_{qc}^2(fg)(x^0;S,T)$ is perfectly accurate in this case. Comparing to the nested-set Hessian $\nabla_s^2(fg)(x^0;S,T),$ we see that it is a major improvement. Indeed, the error bound  for the nested-set Hessian in \eqref{eq:ebgsh} reduces to zero whenever $F=fg$ is a quadratic function.\par
It is worth mentioning that if a minimal poised set for NSHC  $\m(x^0;S,U_k)$ is used, then we know that $S$ and $U_k$ have full row rank. We also know by Proposition \ref{prop:minimalsetispoisedforqi} that any minimal poised set for NSHC of the form $\m(x^0;S,U_k)$ is poised for quadratic interpolation. Therefore, the error bounds for both techniques, that is the simplex calculus Hessian and the quadratic calculus Hessian, are valid (assuming the assumptions on $f$ and $g$ presented in Proposition \ref{prop:productrule} are satisifed).

\begin{cor}[Quotient Rule error bound]\label{prop:quotientrule}
Let $f,g:\R^n\to\R$ be $\mathcal{C}^3$ on $B(x^0;\overline{\Delta})$, where $x^0\in\dom f\cap\dom g$ is the point of interest and $g(x^0)\neq0$. Let $S=[s^1~~s^2~~\cdots~~s^m]\in\R^{n\times m}$ and $T=[t^1~~t^2~~\cdots~~t^k]\in\R^{n\times k}$ be full row rank. Let $B(x^0+s^i;\Delta_T)\subset B(x^0;\overline{\Delta})$ for all $i\in\{1,2,\ldots,m\}$. Then
\begin{align}\label{quotienterror}
&\left\|\nabla^2_{sc}\left(\frac{f}{g}\right)(x^0;S,T)-\nabla^2\left(\frac{f}{g}\right)(x^0)\right\|, \notag \\\leq&\frac{1}{|g^3(x^0)|}\left[E_{\nabla^2_s f}|g^2(x^0)|+E_{\nabla^2_s g}|f(x^0)g^2(x^0)|+2M^s_{\frac{f}{g}}(|f(x^0)|+|g(x^0)|)\right]\Delta_u, 
\end{align}
where\begin{equation*}
M^s_{\frac{f}{g}}=\min\left\{\begin{array}{l}E_{\nabla_s f}E_{\nabla_s g}\Delta_u+E_{\nabla_s g}\|\nabla f(x^0)\|+E_{\nabla_s f}\|\nabla g(x^0)\|,\\E_{\nabla_s f}\|\nabla_s g(x^0;T)\|+E_{\nabla_s g}\|\nabla f(x^0)\|,\\E_{\nabla_s g}\|\nabla_s f(x^0;T)\|+E_{\nabla_s f}\|\nabla g(x^0)\|,\\E_{\nabla^2_s g}\Delta_u+2E_{\nabla_s g}\|\nabla g(x^0)\|, \\E_{\nabla_s g}\|\nabla_s g(x^0;T)\|+E_{\nabla_s g}\|\nabla g(x^0)\|\|\\\end{array}\right\}.
\end{equation*}
Moreover, 
\begin{align}\label{eq:quotienterrorqch}
&\left\|\nabla^2_{qc}\left(\frac{f}{g}\right)(x^0;S,T)-\nabla^2\left(\frac{f}{g}\right)(x^0)\right\|,\notag \\\leq&\frac{1}{|g^3(x^0)|}\left[E_{\nabla^2_s f}|g^2(x^0)|+E_{\nabla^2_s g}|f(x^0)g^2(x^0)|+2M^q_{\frac{f}{g}}(|f(x^0)|+|g(x^0)|)\Delta_u\right]\Delta_u, 
\end{align}
where\begin{equation*}
M^q_{\frac{f}{g}}=\min\left\{\begin{array}{l}
    E_{\nabla Q_f}E_{\nabla Q_g}\Delta_u^2+E_{\nabla Q_g}\|\nabla f(x^0)\|+E_{\nabla Q_f}\|\nabla g(x^0)\|,\\
    E_{\nabla Q_f}\|\nabla Q_g(x^0;S,T)\|+E_{\nabla Q_g}\|\nabla f(x^0)\|,\\
    E_{\nabla Q_g}\|\nabla Q_f(x^0;S,T)\|+E_{\nabla Q_f}\|\nabla g(x^0)\|,\\
    E_{\nabla^2_s g}+2E_{\nabla Q_g}\|\nabla g(x^0)\|, \\
    E_{\nabla Q_g}\|\nabla Q_g(x^0;S,T)\|+E_{\nabla Q_g}\|\nabla g(x^0)\|
    \end{array}\right\}.
\end{equation*}
\end{cor}
Once again, we  see from \eqref{quotienterror} that if $f$ and $g$ are affine functions, then $\nabla_{sc}^2 \left(\frac{f}{g}\right)(x^0;S,T)$ is perfectly accurate. From \eqref{eq:quotienterrorqch}, we see that if $f$ and $g$ are quadratic functions, then  $\nabla_{qc}^2\left(\frac{f}{g}\right)(x^0;S,T)$ is perfectly accurate.
\begin{cor}[Power Rule error bound]\label{prop:powerrule}Let $f:\R^n\to\R$ be $\mathcal{C}^3$ on $B(x^0;\overline{\Delta})$, where $x^0\in\dom f$ is the point of interest. Let $S=[s^1~~s^2~~\cdots~~s^m]\in\R^{n\times m}$  and $T=[t^1~~t^2~~\cdots~~t^k]\in\R^{n\times k}$ be full row rank. Let $B(x^0+s^i;\Delta_T)\subset B(x^0;\overline{\Delta})$ for all $i\in\{1,2,\ldots,m\}$. Then for any $p\in\N\setminus\{1\}$,
\begin{equation}\label{powererror}
\|\nabla^2_{sc} f^p(x^0;S,T)-\nabla^2f^p(x^0)\|\leq\left(pE_{\nabla^2_s f}|f^{p-1}(x^0)|+p(p-1)E_{\nabla_s f}|f^{p-2}(x^0)|M_{f^p}\right)\Delta_u,
\end{equation}
where
\begin{equation*}
M^q_{f^p}=\min\{E_{\nabla_s f}\Delta_u+2\|\nabla f(x^0)\|,\|\nabla_s f(x^0)\|+\|\nabla f(x^0)\|\}.
\end{equation*}
Moreover,
\begin{equation*}
\|\nabla^2_{qc} f^p(x^0;S,T)-\nabla^2f^p(x^0)\|\leq\left(p E_{\nabla^2_sf}|f^{p-1}(x^0)|+p(p-1) 
    E_{\nabla Q_f}|f^{p-2}(x^0)|M_{f^p} \Delta_u\right)\Delta_u,
\end{equation*}
where
\begin{equation*}
M_{f^p}=\min\{E_{\nabla Q_f}\Delta_u^2+2\|\nabla f(x^0)\|,\|\nabla Q_f(x^0)\|+\|\nabla f(x^0)\|\}.
\end{equation*}
\end{cor}
Again we see that if $f$ is a quadratic function, then $\nabla_{qc}^2f^p(x^0;S,T)$ is perfectly accurate.
 Note that the bounds for the simplex calculus Hessian (\eqref{producterror}, \eqref{quotienterror} and\eqref{powererror}) do not necessarily  equal zero when the functions involved are quadratic. Indeed, the bounds will equal zero if all functions involved are affine. For this reason, defining the generalized simplex calculus Hessian as in Table \ref{table:ruletable3} improves the approximation technique greatly.

\section{Conclusion}\label{sec:conc}
The nested-set Hessian provides a compact and simple formula for approximating the Hessian of a function $f$. We have seen that if a minimal poised set for NSHC is used, the number of  distinct function evaluations required  in the computation of the nested-set Hessian is $(n+1)(n+2)/2.$ We investigated the relationship between a minimal poised set for NSHC   and poisedness for quadratic interpolation, proving that a minimal poised set for NSHC of the form $\m(x^0;S,U_k)$ is well-poised for quadratic interpolation, but that the converse does not necessarily hold. In Section \ref{sec:qishc}, we also developed formulae for obtaining the parameters of the quadratic interpolation function of $f$ over a minimal poised set for NSHC of the form $\m(x^0;S,U_k)$. Building this quadratic interpolation does not require any new function evaluations, as it uses the same function evaluations involved in the computation of the nested-set Hessian. In Section \ref{sec:error}, we developed an error bound for the nested-set Hessian. The error bound shows that the nested-set Hessian is an order-1  Hessian accuracy approximation depending on $\Delta_u$, the maximal radius of the two sets involved. Analyzing the error bound, we  saw that if $f$ is quadratic, then the nested-set Hessian returns the true Hessian $\nabla^2f$. Finally, we presented two calculus-based Hessian approximation techniques in Section \ref{sec:calculus}. Both approaches are order-1 Hessian accuracy approximation depending on $\Delta_u$. We saw that the approximation technique named quadratic calculus Hessian requires a little more work, since the quadratic interpolation function  must be constructed. In return, the technique is perfectly accurate if the functions involved are quadratic, which is not necessarily the case for the simplex calculus Hessian.

An obvious direction for the advancement of this research is the application of the Hessian approximation techniques within model-based DFO algorithms. Now that we have  a calculus-based approximation technique  of the gradient (the technique was named simplex calculus gradient and introduced in \cite{MR4074016}) and a calculus-based approximation of the Hessian, it is now possible to use these techniques in advanced algorithms  such as the model-based trust region method \cite[Chpt 11]{audet2017derivative}. For instance, for an optimization  problem involving the product of $k$ functions,  a model-based trust region algorithm using the calculus-based techniques could be compared to a model-based trust region algorithm that does not use any calculus rules.

Future research should also determine if the sets $U_0, U_1, \dots, U_n$ as defined in Proposition \ref{prop:reducingfe} are the only possible choices for $T \in \R^{n \times n}$ such that $\s(x^0;S,T),$ where $S \in \R^{n \times n}$ is full rank,  is a minimal poised set for NSHC at $x^0$.  It can be proven that it is indeed the case by using brute force in $\R$ and $\R^2,$ but it is still unclear how to generalized this claim in an arbitrary dimension $n$.
\bibliographystyle{plain.bst}
\bibliography{Bibliography}

\def\cprime{$'$}
\begin{thebibliography}{10}

\bibitem{amaioua2018efficient}
N.~Amaioua, C.~Audet, A.~Conn, and S.~Le~Digabel.
\newblock Efficient solution of quadratically constrained quadratic subproblems
  within the mesh adaptive direct search algorithm.
\newblock {\em European Journal of Operational Research}, 268(1):13--24, 2018.

\bibitem{audet2008nonsmooth}
C.~Audet, V.~B{\'e}chard, and S.~Le~Digabel.
\newblock Nonsmooth optimization through mesh adaptive direct search and
  variable neighborhood search.
\newblock {\em Journal of Global Optimization}, 41(2):299--318, 2008.

\bibitem{audet2017derivative}
C.~Audet and W.~Hare.
\newblock {\em Derivative-free and {B}lackbox {O}ptimization}.
\newblock Springer, 2017.

\bibitem{audet2018algorithmic}
C.~Audet and W.~Hare.
\newblock Algorithmic construction of the subdifferential from directional
  derivatives.
\newblock {\em Set-Valued and Variational Analysis}, 26(3):431--447, 2018.

\bibitem{audet2018mesh}
C.~Audet and C.~Tribes.
\newblock Mesh-based {N}elder--{M}ead algorithm for inequality constrained
  optimization.
\newblock {\em Computational Optimization and Applications}, 71(2):331--352,
  2018.

\bibitem{AudetIanniLeDigTribes2014}
Charles Audet, Andrea Ianni, S{\'e}bastien Le~Digabel, and Christophe Tribes.
\newblock Reducing the number of function evaluations in mesh adaptive direct
  search algorithms.
\newblock {\em SIAM Journal on Optimization}, 24(2):621--642, 2014.

\bibitem{bagirov2008discrete}
A.~Bagirov, B.~Karas{\"o}zen, and M.~Sezer.
\newblock Discrete gradient method: derivative-free method for nonsmooth
  optimization.
\newblock {\em Journal of Optimization Theory and Applications},
  137(2):317--334, 2008.

\bibitem{BBN2018}
A.~Berahas, R.~Byrd, and J.~Nocedal.
\newblock Derivative-free optimization of noisy functions via quasi-{Newton}
  methods.
\newblock {\em SIAM Journal on Optimization}, 2019.
\newblock To appear.

\bibitem{berghen2005condor}
F.~Berghen and H.~Bersini.
\newblock {CONDOR}, a new parallel, constrained extension of {P}owell's
  {UOBYQA} algorithm: {E}xperimental results and comparison with the {DFO}
  algorithm.
\newblock {\em Journal of Computational and Applied Mathematics},
  181(1):157--175, 2005.

\bibitem{bortz1998simplex}
D.~Bortz and C.~Kelley.
\newblock The simplex gradient and noisy optimization problems.
\newblock In {\em Computational {M}ethods for {O}ptimal {D}esign and
  {C}ontrol}, pages 77--90. Springer, 1998.

\bibitem{cocchi2018implicit}
G.~Cocchi, Giampaolo L., A.~Papini, and M.~Sciandrone.
\newblock An implicit filtering algorithm for derivative-free multiobjective
  optimization with box constraints.
\newblock {\em Computational Optimization and Applications}, 69(2):267--296,
  2018.

\bibitem{conn2008geometry}
A.~Conn, K.~Scheinberg, and L.~Vicente.
\newblock Geometry of interpolation sets in derivative free optimization.
\newblock {\em Math. Program.}, 111(1-2):141--172, 2008.

\bibitem{conn2009introduction}
A.~Conn, K.~Scheinberg, and L.~Vicente.
\newblock {\em Introduction to derivative-free optimization}.
\newblock SIAM, 2009.

\bibitem{MR3935094}
I.~Coope and R.~Tappenden.
\newblock Efficient calculation of regular simplex gradients.
\newblock {\em Computational Optimization and Applications}, 72(3):561--588,
  2019.

\bibitem{MR2457346}
A.~Cust\'{o}dio, J.~Dennis, and L.~Vicente.
\newblock Using simplex gradients of nonsmooth functions in direct search
  methods.
\newblock {\em IMA Journal of Numerical Analysis}, 28(4):770--784, 2008.

\bibitem{Gratton2017b}
S.~Gratton, C.~Royer, and L.~Vicente.
\newblock A decoupled first/second-order steps technique for nonconvex
  nonlinear unconstrained optimization with improved complexity bounds.
\newblock {\em Mathematical Programming}, 2019.
\newblock To appear.

\bibitem{Gratton2017direct}
S.~Gratton, C.~Royer, L.~Vicente, and Z.~Zhang.
\newblock Direct search based on probabilistic feasible descent for bound and
  linearly constrained problems.
\newblock {\em Computational Optimization and Applications}, 72(3):525--559,
  2019.

\bibitem{MR4074016}
W.~Hare and G.~Jarry-Bolduc.
\newblock Calculus identities for generalized simplex gradients: rules and
  applications.
\newblock {\em SIAM Journal on Optimization}, 30(1):853--884, 2020.

\bibitem{hare2020error}
W.~Hare, G.~Jarry-Bolduc, and C.~Planiden.
\newblock Error bounds for overdetermined and underdetermined generalized
  centred simplex gradients.
\newblock {\em arXiv preprint arXiv:2006.00742}, 2020.

\bibitem{hare2013derivative}
W.~Hare and J.~Nutini.
\newblock A derivative-free approximate gradient sampling algorithm for finite
  minimax problems.
\newblock {\em Computational Optimization and Applications}, 56(1):1--38, 2013.

\bibitem{hare2016proximal}
W.~Hare, C.~Sagastiz{\'a}bal, and M.~Solodov.
\newblock A proximal bundle method for nonsmooth nonconvex functions with
  inexact information.
\newblock {\em Computational Optimization and Applications}, 63(1):1--28, 2016.

\bibitem{hare2020discussion}
Warren Hare.
\newblock A discussion on variational analysis in derivative-free optimization.
\newblock {\em Set-Valued and Variational Analysis}, pages 1--17, 2020.

\bibitem{hare2019derivative}
Warren Hare, Chayne Planiden, and Claudia Sagastiz{\'a}bal.
\newblock A derivative-free {VU}-algorithm for convex finite-max problems.
\newblock {\em Optimization Methods and Software}, 35(3):521--559, 2020.

\bibitem{larson2016manifold}
J.~Larson, M.~Menickelly, and S.~Wild.
\newblock Manifold sampling for $l_1$ nonconvex optimization.
\newblock {\em SIAM Journal on Optimization}, 26(4):2540--2563, 2016.

\bibitem{liuzzi2019trust}
G.~Liuzzi, S.~Lucidi, F.~Rinaldi, and L.~Vicente.
\newblock Trust-region methods for the derivative-free optimization of
  nonsmooth black-box functions.
\newblock {\em SIAM Journal on Optimization}, 29(4):3012--3035, 2019.

\bibitem{Maggiar2018}
A.~Maggiar, A.~W\"{a}chter, I.~Dolinskaya, and J.~Staum.
\newblock A derivative-free trust-region algorithm for the optimization of
  functions smoothed via {Gaussian} convolution using adaptive multiple
  importance sampling.
\newblock {\em SIAM Journal on Optimization}, 28(2):1478--1507, 2018.

\bibitem{MMSMW2017}
M.~Menickelly and S.~Wild.
\newblock Derivative-free robust optimization by outer approximations.
\newblock {\em Mathematical Programming}, 2019.
\newblock To appear.

\bibitem{powell2003trust}
M.~Powell.
\newblock On trust region methods for unconstrained minimization without
  derivatives.
\newblock {\em Mathematical Programming}, 97(3):605--623, 2003.

\bibitem{powell2009bobyqa}
M.~Powell.
\newblock The {BOBYQA} algorithm for bound constrained optimization without
  derivatives.
\newblock {\em Cambridge NA Report NA2009/06, University of Cambridge,
  Cambridge}, pages 26--46, 2009.

\bibitem{MR3348587}
R.~Regis.
\newblock The calculus of simplex gradients.
\newblock {\em Optimization Letters}, 9(5):845--865, 2015.

\bibitem{rockwets}
R.~Rockafellar and R.~Wets.
\newblock {\em Variational analysis}.
\newblock Grundlehren der Mathematischen Wissenschaften [Fundamental Principles
  of Mathematical Sciences]. Springer-Verlag, Berlin, 1998.

\bibitem{shashaani2018astro}
S.~Shashaani, F.~Hashemi, and R.~Pasupathy.
\newblock {ASTRO-DF}: {A} class of adaptive sampling trust-region algorithms
  for derivative-free stochastic optimization.
\newblock {\em SIAM Journal on Optimization}, 28(4):3145--3176, 2018.

\bibitem{verderio2017construction}
A.~Verd{\'e}rio, E.~Karas, L.~Pedroso, and K.~Scheinberg.
\newblock On the construction of quadratic models for derivative-free
  trust-region algorithms.
\newblock {\em EURO Journal on Computational Optimization}, 5(4):501--527,
  2017.

\bibitem{wild2013global}
S.~Wild and C.~Shoemaker.
\newblock Global convergence of radial basis function trust-region algorithms
  for derivative-free optimization.
\newblock {\em SIAM REVIEW}, 55(2):349--371, 2013.

\end{thebibliography}
\end{document}